\documentclass[oneside,english]{amsart}

\usepackage[latin1]{inputenc}

\usepackage{amssymb}

\usepackage{amssymb}
\usepackage[]{epsf,epsfig,amsmath,amssymb,amsfonts,latexsym}

\newtheorem{thm}{Theorem}[section]
\newtheorem{lem}[thm]{Lemma}
\newtheorem{prop}[thm]{Proposition}
\newtheorem{cor}[thm]{Corollary}

\theoremstyle{definition}

\newtheorem{remark}[thm]{Remark}
\newtheorem{example}[thm]{Example}
\newtheorem{question}[thm]{Question}

\newcommand{\AAA}{\mathcal{A}}
\newcommand{\BBB}{\mathcal{B}}
\newcommand{\ZZ}{\mathbb{Z}}
\newcommand{\NN}{\mathbb{N}}
\newcommand{\RR}{\mathbb{R}}
\newcommand{\ZD}{\mathbb{Z}^d}

\newcommand{\act}[2]{{#1} \curvearrowright {#2}}

\newcommand{\TotOrd}{\mathit{TotalOrd}}
\newcommand{\Ord}{\mathit{Ord}}

\newcommand{\Prob}{\mathit{Prob}}

\newcommand{\expect}{\mathbb{E}}  
\newcommand{\Borel}{\mathit{Borel}}

\newcommand{\hRok}{h^{\mathit{Rok}}}

\title{Predictability, topological entropy and invariant random orders}


\author[1]{Andrei Alpeev, Tom Meyerovitch and Sieye Ryu}
\address{Andrei Alpeev,
	Chebyshev Laboratory, St. Petersburg State University, 14th Line, 29b, Saint Petersburg, 199178 Russia}
\email{a.alpeev@spbu.ru}
\address{Tom Meyerovitch,
	Department of Mathematics, Ben Gurion University of the Negev, P.O.B. 653
	Be'er Sheva 8410501, Israel.}
\email{mtom@math.bgu.ac.il}
\address{Sieye Ryu,
	Institute of Mathematics and Statistics, University of Sao Paulo, Rua do Matao 1010,  CEP 05508-090, Sao Paulo, Brazil.}
\email{sieyeryu@ime.usp.br}

\begin{document}


\maketitle
\begin{abstract}   
	We  prove  that a topologically predictable action of a countable amenable group has zero topological entropy, as conjectured by Hochman. On route,  we investigate  invariant random orders and formulate a unified Kieffer-Pinsker formula for the Kolmogorov-Sinai entropy of measure preserving actions of amenable groups. 
	We also present a proof  due to  Weiss  for the fact that topologically prime actions of sofic groups have non-positive topological sofic entropy. 
\end{abstract}

The aim of this note is to prove the following:


\begin{thm}\label{thm:hochman_predictablilty_zero_entropy}
Let $\Gamma$  be a countable amenable group that acts on
a compact metric space $X$ by homeomorphism,  and let $S \subset \Gamma$ be a subsemigroup of $\Gamma$ with $1_\Gamma \not\in S$.
If every continuous function $f \in C(X)$ is contained in the closed algebra generated by $\{f \circ s :~ s \in S\}$ and the constant functions, then the action $\act{\Gamma}{X}$ has zero topological entropy.
\end{thm}

Theorem \ref{thm:hochman_predictablilty_zero_entropy} was initially proved in the case $\Gamma=  \ZZ$ and $S = \ZZ_+$ by Kamin\'ski, Siemaszko and Szyman\'ski in \cite{MR2025310} (see also \cite{MR2162751} and  \cite{MR2186250}). In \cite{MR2873161} Hochman gave another proof and also generalized this to the case $\Gamma = \ZD$.
In the same work Hochman conjectured Theorem \ref{thm:hochman_predictablilty_zero_entropy}. 
Later, Huang, Jin and Ye in \cite{MR3530053} proved Theorem \ref{thm:hochman_predictablilty_zero_entropy}  under the additional assumption  that  $\Gamma$ is torsion-free and locally nilpotent.
It turns out that  the proof of Theorem \ref{thm:hochman_predictablilty_zero_entropy} does not involve any considerable new ideas or tools beyond those developed by Hochman in the same paper where the question had been posed. However, as suggested to us by L. Bowen, without paying too great a price, we are able to obtain more general results about predictability in the presence of  \emph{invariant random orders} on groups.
We will deduce Theorem \ref{thm:hochman_predictablilty_zero_entropy} above as a particular case of the slightly more general Theorem \ref{thm:predicatable_implies_zero_entropy_amenable} below.

For a countable group $\Gamma$ we denote by $\Ord(\Gamma) \subset \{0,1\}^{\Gamma \times \Gamma}$ the space of all orders on $\Gamma$.
The space $\Ord(\Gamma)$ is metrizable and compact, and admits a natural $\Gamma$-action.
(see Section \ref{sec: random orderings}).
Let $\nu$ be a $\Gamma$-invariant measure on $\Ord(\Gamma)$.
We will say that action $\act{\Gamma}{X}$ is \emph{$\nu$-topologically predictable relative to a topological factor map  $\pi:X \to Y$} if the following holds: For any $f \in C(X)$, and $\nu$-a.e. $\prec \in \Ord(\Gamma)$ the function $f$ is contained in the closed algebra generated by $\{ f\circ g: g \prec 1_\Gamma\}$ together with the image of  $C(Y)$ in $C(X)$ under the map  $\pi_*:C(Y) \to C(X)$.

\begin{thm}\label{thm:predicatable_implies_zero_entropy_amenable}
	Let $\Gamma$  be a countable amenable group that acts on
	two compact metric spaces $X,Y$ by homeomorphisms with a   continuous $\Gamma$-equivariant map $\pi:X \to Y$. 
	If there exists
	 a $\Gamma$-invariant probability measure $\nu$ on the space $\Ord(\Gamma)$  such that  
	$\act{\Gamma}{X}$ is $\nu$-topologically predictable relative to  $\pi:X \to Y$, then the topological entropy of the action $\act{\Gamma}{X}$ is equal to the topological entropy of the action $\act{\Gamma}{Y}$.
\end{thm}

Theorem \ref{thm:hochman_predictablilty_zero_entropy} is a special case of Theorem \ref{thm:predicatable_implies_zero_entropy_amenable}, obtained by taking $\act{\Gamma}{Y}$ to be the trivial action on a singleton, and taking $\nu$ to be the delta measure concentrated on the order $\prec_S$ given by
\begin{equation*} 
g_1 \prec_S g_2 ~ \Leftrightarrow ~ g_1g_2^{-1} \in S.
\end{equation*}

In the last decade entropy theory has emerged for actions of non-amenable groups. Entropy for measure preserving  actions of sofic group was developed starting with the seminal paper \cite{MR2552252}. Sofic topological entropy was introduced by Kerr and Li in \cite{MR2854085}. The reader may find more details in \cite{1711.02062},  \cite{1704.06349} and \cite{MR3616077}. It is natural to ask if the above results extend to the non-amenable setting.

In Section \ref{sec:prime_actions} we include a short proof that was communicated to us by Benjy Weiss for the fact that topologically prime actions have zero entropy. Weiss's proof  uses some  similar techniques as the result about predictable systems. The result about prime systems  also applies to actions of sofic groups, with appropriate adjustments. 

{\bf Acknowledgements.} We thank Benjy Weiss for sharing his argument regarding the entropy of topologically prime systems and Lewis Bowen for insightful comments and references regarding random invariant orders. We also thank Ben Hayes for his comments on the first version of this paper. Andrei Alpeev was supported by ``Native towns'', a social investment program of PJSC ``Gazprom Neft''. Tom Meyerovitch and Sieye Ryu acknowledge support by the  Israel Science Foundation (grants no. 626/14 and 1052/18) and the The Center For Advanced Studies In Mathematics in Ben Gurion University. Tom Meyerovitch thanks the Pacific Institute for Mathematical Studies and the department of mathematics at the University of British Columbia for warm hospitality.  

\section{Preliminaries}\label{Sec:prelim}

In this paper $\act{\Gamma}{X}$ will denote  a left action of a countable group $\Gamma$ by homeomorphisms on a compact metric space $X$.
We will denote by $\Prob(X)$ the (compact, convex) space of Borel probability measures on $X$, and by $\Prob_\Gamma(X)$ the subset of $\Gamma$-invariant  Borel probability measures.
We will use the notation $\act{\Gamma}{(X,\mu)}$ to indicate that $\mu \in \Prob_\Gamma(X)$, and in this case we will say that the $\Gamma$ action on $(X,\mu)$ is measure preserving.
For two partitions $\alpha$ and $\beta$ denote their join by $\alpha \vee \beta$. 
Similarly, for two $\sigma$-algebras $\AAA$ and $\mathcal{B}$ denote by $\AAA \vee \mathcal{B}$ the smallest $\sigma$-algebra containing both.
We will say that a partition $\beta$ is finer than a partition $\alpha$ if each element of $\alpha$ is a union of elements from $\beta$.

Let $\alpha$ be a countable measurable  partition of a standard probability space $X$. We denote $H_\mu(\alpha)$ its Shannon entropy; for two partitions $\alpha$ and $\beta$ we denote $H_\mu(\alpha \vert \beta)$ the Shannon entropy of partition $\alpha$ relative to partition $\beta$. For a partition $\alpha$ and a $\sigma$-subalgebra $\AAA$ of a standard probability space, the Shannon entropy of $\alpha$ relative to $\AAA$ is given by
$$H_\mu(\alpha \vert \AAA) = -\int \log \mu \left( \alpha \mid \AAA\right)d\mu.$$ 
 If $\alpha$ is a partition of finite Shannon entropy and $\AAA$ is a $\sigma$-subalgebra, then the following holds:
\[
H_\mu(\alpha \vert \AAA ) = \inf \lbrace  H_\mu(\alpha \vert \beta), \text{ where } \beta \subseteq \AAA, H_\mu(\beta)<+\infty \rbrace;
\]
Moreover, if  $(\beta_i)_{i=1}^\infty$ is a sequence of partitions having finite Shannon entropy and
$\AAA$ is the smallest $\sigma$-subalgebra that contains each $\beta_i$ then
\[
H_\mu(\alpha \vert \AAA) = \lim_{i \to \infty} H_\mu(\alpha \vert \bigvee_{j \le i} \beta_j).
\]

We recall that countable group $\Gamma$ is  {\em amenable} if it has a F\o lner sequence, namely a sequence $(F_i)_{i=1}^\infty$ of finite subsets such that for any $g \in \Gamma$ the following holds:
\[
\lim_{i \to \infty } \frac{\lvert g F_i \setminus F_i \rvert}{\lvert F_i\rvert} = 0.
\]

It is a well known fact that a countable group  $\Gamma$ is amenable if and only if
any action of $\Gamma$ by continuous affine transformations on a compact convex subset of a separable locally convex topological vector space has a fixed point (for this and  many other equivalent conditions for amenability, see for instance \cite{MR767264}).

  For a countable  partition  $\alpha = \lbrace B_1, B_2, \ldots\rbrace$ and $g \in \Gamma$ we denote 
$$\alpha^g  = \lbrace g^{-1}(B_1), g^{-1}(B_2), \ldots\rbrace.$$
If   $F \subset \Gamma$ is finite let $\alpha^F = \bigvee_{g \in F} \alpha^g$.
This is again a countable partition.
If $F \subset \Gamma$ is infinite, we let $\alpha^F$ denote 
  the smallest $\sigma$-subalgebra containing all $\alpha^g$ for $g \in F$. 
  
  Now suppose that $(F_i)_{i=1}^\infty$ is a F\o lner sequence for the group $\Gamma$.
  We will denote the  {\em Kolmogorov-Sinai entropy}  of the partition $\alpha$ relative to a $\Gamma$-invariant $\sigma$-subalgebra $\AAA \subset \Borel(X)$ by $h_{\Gamma}(\alpha, X,\mu \vert \AAA)$. This  is defined by the formula
\[h_{\Gamma}(\alpha, X,\mu \vert \AAA)=
\lim_{i \to \infty} \frac{H_\mu(\alpha^{F_i} \vert \AAA)}{\lvert F_i\rvert}.
\]
It is known that $h_{\Gamma}(\alpha, X,\mu \vert \AAA)$ does not depend on the choice of the F\o lner sequence. The {\em Kolmogorov-Sinai entropy} of the action relative to a $\sigma$-subalgebra $\AAA$ will be denoted by $h_{\Gamma}(X,\mu \vert \AAA)$. This is 
given by
$$ h_{\Gamma}(X,\mu \vert \AAA) = \sup \lbrace h_\Gamma(\alpha, X, \mu \vert \AAA) :~ H_\mu(\alpha)<\infty \rbrace.$$
We denote  the Kolmogorov-Sinai entropy of the factor corresponding to $\AAA$ by 

\[
h_\Gamma(\AAA, X,\mu) = \sup \lbrace h_\Gamma(\alpha, X, \mu) :~ \alpha \subset \AAA, H_\mu(\alpha)<\infty \rbrace.
\]
For an invariant $\sigma$-subalgebra $\BBB$ we denote $h_\Gamma(\BBB, X, \mu \vert \AAA)$ the entropy of the factor corresponding to $\BBB$ relative to $\AAA$.

The Kolmogorov-Sinai theorem asserts that 
\[
h_\Gamma(\alpha^\Gamma \vee \AAA , X, \mu \vert \AAA) = h_\Gamma(\alpha , X, \mu \vert \AAA).
\]

The topological entropy of an action $\act{\Gamma}{X}$ of a countable amenable group $\Gamma$ on a compact space $X$ will be denoted by $h_\Gamma(X)$.


\section{Invariant random orders and invariant random pasts}\label{sec: random orderings}

In the sequel we will employ some rudimentary theory of random orders on groups. Random orders were successfully used in \cite{MR3231224} and \cite{MR2286034} to prove results concerning deterministic orders on amenable groups.
 See the book \cite{1408.5805} for background and much more. Particular  invariant random orders have been applied to entropy theory, notably Keiffer's paper \cite{MR0393422} about actions of amenable groups, and  \cite{1705.08559,austin-podder,1602.06680} for actions of countable but not necessarily amenable groups.  
 
 Consider the set $2^{\Gamma \times \Gamma}$ of binary relations on a countable group $\Gamma$, endowed with the product topology.  This topology makes $2^{\Gamma \times \Gamma}$ a compact  metrizable space.
 We will consider the left action $\act{\Gamma}{2^{\Gamma \times \Gamma}}$  given by
 $$x (g \cdot R)y \Leftrightarrow  (xg) R (yg) \mbox{ for } x,y, g \in \Gamma \mbox{ and }R \in 2^{\Gamma \times \Gamma}.$$

 Recall that a relation $\prec \in 2^{\Gamma \times \Gamma}$ on $\Gamma$ is called a {\em partial order} if the following requirements hold:
 \begin{enumerate}
 	\item It is antisymmetric, which means  that  if $x \prec y$ then  $y \not\prec x$.
 	\item It is transitive, which means that if $x \prec y$ and $y \prec z$ then $x \prec z$.
 \end{enumerate}
 A partial order $\prec$ is called {\em total}  if for any $x,y \in \Gamma$ either $x \prec y$, $y \prec x$ or $x=y$.
 Let $\Ord(\Gamma)$ denote the set of all partial orders on $\Gamma$. Denote  the set of all total orders on $\Gamma$ by $\TotOrd(\Gamma)$. It is not hard to see that both $\TotOrd(\Gamma)$ and $\Ord(\Gamma)$ are closed subsets of $2^{\Gamma \times \Gamma}$. 
 It is easy to see that $\Ord(\Gamma)$ and $\TotOrd(\Gamma)$ are $\Gamma$-invariant subsets.
 
 To a partial order $\prec$ we associate the \emph{past} (at the identity):
 
 \begin{equation*}\label{eq:Phi_prec}
 \Phi_\prec = \left\{\gamma \in\Gamma ~:~ 1_\Gamma \prec \gamma  \right\}.
 \end{equation*} 
 A  fixed point for the action $\act{\Gamma}{\Ord(\Gamma)}$ is called a (deterministic) invariant  order on $\Gamma$.
 If $\prec$ is an  invariant order, then it is straightforward to check that the associated past $\Phi_\prec \subset \Gamma$ is a semigroup that does not contain the identity $1_\Gamma$. If  $\prec \in\Ord(\Gamma)$ is a (deterministic) invariant \emph{total} order, then the  associated past $\Phi_\prec \subset \Gamma$ is an \emph{algebraic past} for $\Gamma$, namely it is a semigroup with the property that 
 \begin{equation*}
 \Gamma = \Phi_\prec \uplus \Phi_\prec^{-1} \uplus \{1_\Gamma\}.
 \end{equation*} 
 A group $\Gamma$ that admits a deterministic invariant total order (or equivalently, admits an algebraic past) is called \emph{left-orderable}.
 An \emph{invariant random order} on $\Gamma$ is a  $\Gamma$-invariant Borel probability measure on $\Ord(\Gamma)$.
 An \emph{invariant random total order} is a $\Gamma$-invariant Borel probability measure on $\TotOrd(\Gamma)$. Equivalently, it is an invariant random order that is supported on the set of total orders.
 Thus, consistently with our notation the space of invariant random orders will be denoted by
 $$ \Prob_\Gamma(\Ord(\Gamma)),$$
 and the space of invariant random total orders will be denoted by
 $$\Prob_\Gamma(\TotOrd(\Gamma)).$$
 In the sequel we will use the probabilistic convention and write ``$\prec$ is an invariant random total order with law $\nu$''
  to mean that $\prec$ is an $\Ord(\Gamma)$-valued random variable with distribution $\nu$, where 
 $$ \nu \in \Prob_\Gamma(\Ord(\Gamma)).$$
 In this case for $F \in L^1(\nu)$ we will use the notation
 $$\expect_{\prec} F(\prec) = \int F(\prec)d\nu(\prec).$$
 
 In contrast to deterministic invariant orders, every countable group $\Gamma$ admits at least one invariant random total order.
 Namely, consider the random process $(\xi_\gamma)_{\gamma \in \Gamma}$ of independent random variables such that each $\xi_\gamma$ has uniform distribution on $[0,1]$. Then each realization of this process induces an order on $\Gamma$. 

 We now define what it means for one invariant random total order to extend another:
 Let $\nu,\tilde \nu \in \Prob_\Gamma(\Ord(\Gamma))$ be invariant random  orders on $\Gamma$. Recall that a \emph{joining}  of $\nu$ and $\tilde \nu$ is a probability measure    $\lambda \in \Prob(\Ord(\Gamma)\times \Ord(\Gamma))$ that is invariant under the $\Gamma$-action on the product space and has the property that push-forward of the projection onto the first coordinate coincides with $\nu$ and the  push-forward of the projection of $\lambda$ onto the  second coordinate coincides with $\tilde \nu$.
 We say that an invariant random order $\tilde \nu$ extends $\nu$ if there exists a joining $\lambda$ of $\nu$ and $\tilde \nu$
 with the property that
  
 \begin{equation}\label{eq:ext_nu} \lambda\left( \left\{ 
 (\prec,\tilde \prec) \in \Ord(\Gamma) \times \Ord(\Gamma)~:~ 
 x \prec y \Rightarrow x \tilde \prec y
 \right\}
 \right) =1
 \end{equation}
  \begin{lem}\label{lem: good order}
 	Let  $\Gamma$ be an amenable group. Then any invariant  random order on $\Gamma$ 
 	can be extended to an invariant  random total order. 
 \end{lem}
 \begin{proof}
 Let $\nu \in \Prob_\Gamma(\Ord(\Gamma))$ be an invariant random order.
 Consider the set 
 $J(\nu)$ that consists 
 of Borel probability measures  $\lambda \in \Prob(\Ord(\Gamma) \times \TotOrd(\Gamma))$ whose push-forward via the projection onto the first coordinate is equal to $\nu$ and have the property that \eqref{eq:ext_nu} holds.
 Then $J(\nu)$ is a non-empty, compact convex set. Because $\nu$ is a $\Gamma$-invariant probability measure, the set $J(\nu)$ is also  invariant under the natural  
 action of  $\Gamma$. By amenability of  $\Gamma$, the  action $\act{\Gamma}{J(\nu)}$ by affine transformations admits a fixed point $\lambda \in J(\nu)$. It follows that any such fixed point is a joining  of $\nu$ with some invariant random total order $\tilde \nu$ that extends $\nu$. 
 \end{proof}
 
For torsion free locally nilpotent groups, the Rhemtulla-Formanek Theorem \cite{MR0327605,MR0311538} asserts that any \emph{deterministic} invariant order  extends to a \emph{deterministic} invariant total order. Equivalently, for this class of groups any sub-semigroup that does not contain the identity extends to an algebraic past. 
The Rhemtulla-Formanek theorem was  used in \cite{MR3530053} to prove Theorem \ref{thm:hochman_predictablilty_zero_entropy} for the class of torsion free locally nilpotent groups. Examples provided in the same paper show that the conclusion of the Rhemtulla-Formanek Theorem fails for more general groups, including  some left-orderable amenable ones.
Lemma \ref{lem: good order} can be viewed as an easy ``random substitute'' for the Rhemtulla-Formanek theorem that applies to all amenable groups
 \begin{question}
 	Does the statement of Lemma \ref{lem: good order} hold without the amenability assumption on the group?
 \end{question}

Let  $\nu \in \Prob_\Gamma(\Ord(\Gamma))$ be an invariant random order.
Recall that the action $\act{\Gamma}{X}$ is \emph{topologically $\nu$-predictable} relative  to $\pi:X \to Y$ if for any $f \in C(X)$, and $\nu$-a.e. $\prec \in \Ord(\Gamma)$ the function $f$ is contained in the closed algebra generated by $\{ f\circ g: g \prec 1_\Gamma\}$ together with the image of  $C(Y)$ in $C(X)$ under the map  $\pi_*:C(Y) \to C(X)$.

Now 
suppose $\mu \in \Prob_\Gamma(X)$ is a $\Gamma$-invariant probability measure for the action $\act{\Gamma}{(X,\mu)}$ and 
 $\nu \in \Prob_\Gamma(\Ord(\Gamma))$ is an invariant  random order. We say that the measure preserving action $\act{\Gamma}{(X,\mu)}$ is \emph{measure-theoretically $\nu$-predictable relative to $\pi:X \to Y$} if for every   countable  Borel partition $\alpha$ with $H_\mu(\alpha) < \infty$ 
we have that for $\nu$-a.e $\prec \in \Ord(\Gamma)$, the partition  $\alpha$ is measurable with respect to the $\mu$-completion of $\alpha^{\Phi_\prec} \vee \pi^{-1}(\Borel(Y))$.

Let us introduce the following random generalization for the notion of an algebraic past. An \emph{invariant random past} on $\Gamma$ is  a random function $\tilde \Phi:\Gamma \to 2^\Gamma$, or equivalently a 
Borel probability measure  on $(2^\Gamma)^\Gamma$, with the following properties:
\begin{enumerate}
	\item[(i)]  For almost every instance of $\tilde \Phi:\Gamma \to 2^\Gamma$ and for all $g \in \Gamma$ the condition
	$g \not \in \tilde \Phi(g)$ holds.
	\item[(ii)]
	For almost every instance of $\tilde \Phi:\Gamma \to 2^\Gamma$, for all $g,h \in \Gamma$, if $g \in \tilde \Phi(h)$ then $\tilde \Phi(g) \subset \tilde \Phi(h)$. 
		\item[(iii)] If $g \ne h$ then either $g \in \tilde \Phi(h)$ or $h \in \tilde \Phi(g)$.
	\item[(iv)] For all $g \in \Gamma$ the random subsets
	$\tilde \Phi(g)$ and $\tilde \Phi(1_\Gamma)g$ have the same distribution.
\end{enumerate}

It follows directly from the definitions that if  $\prec$ is an invariant random total order, then the random function given by $g \mapsto \{h \in \Gamma :~ h \prec g\}$ defines 
an invariant random past. If $\tilde \Phi$ is a random past on $\Gamma$ with law $\tilde \nu \in \Prob((2^\Gamma)^\Gamma)$ and $F \in L^1(\tilde \nu)$ we use the following probabilistic notation:

 $$\expect_{\tilde \Phi} F(\tilde \Phi)  = \int F(\tilde \Phi) d\tilde \nu(\tilde \Phi).$$

\section{The  Kieffer-Pinsker formula}\label{sec: generalized kieffer}

In this section we state and prove a simultaneous but rather straightforward generalization of  Kieffer's entropy formula \cite{MR0393422}  and of Pinsker's entropy formula for actions of amenable groups \cite[Theorem $3.1$]{MR3296581}. 
The earliest and most basic case of this formula for the group $\Gamma =\ZZ$ with the usual order goes back to Kolmogorov's very first paper \cite{kolmogorov} on entropy. 

\begin{thm}[The Kieffer-Pinsker formula]\label{th: generalized kieffer}
Let $\act{\Gamma}{(X,\mu)}$ be a probability  measure preserving action of a countable amenable group $\Gamma$,  
let $\tilde \Phi:\Gamma \to 2^\Gamma$ be an invariant random past on $\Gamma$ with $\Phi = \tilde \Phi(1_\Gamma) \subset \Gamma$. Suppose that
$\alpha$ is a Borel partition with $H_\mu(\alpha) < +\infty$ and that
 $\AAA$ is a $\Gamma$-invariant $\sigma$-algebra on $X$. Then
	the following holds: 
	\begin{equation*}
	h_\Gamma(\alpha, X, \mu\vert \AAA) = \expect_{\tilde \Phi} H_\mu(\alpha \vert \alpha ^{\Phi} \vee \AAA),
	\end{equation*}
\end{thm}


Before the proof we will establish a couple of auxiliary lemmata. We assume that $\tilde \Phi:\Gamma \to 2^\Gamma$, $\Phi= \tilde \Phi(1_\Gamma)$, $\alpha$ and $\AAA$ are as in the statement of Theorem \ref{th: generalized kieffer}.

\begin{lem}\label{lem:D_approx_H}
	For any $\varepsilon>0$ there is such a finite subset $D$ of $\Gamma$ so that 
	for any $D' \supset D$
	\begin{equation}\label{eq:H_D_prime}
	\expect_{\tilde \Phi}H_\mu(\alpha \vert \alpha^{D' \cap \Phi} \vee \AAA) \leq \expect_{\tilde \Phi}H_\mu(\alpha \vert \alpha^{\Phi} \vee \AAA) + \varepsilon. 
	\end{equation}
\end{lem}
\begin{proof}
	Let us consider an arbitrary subset $\Phi \subset \Gamma$.
	Let $(D_i)_{i=1}^\infty$ be an increasing sequence of finite subsets of $\Gamma$ such that $\bigcup D_i = \Gamma$. 
	A classical argument using  the Martingale convergence theorem and Chung's Lemma (as in  \cite[Theorem $14.28$]{MR1958753}) implies that
	\begin{equation*}
	\lim_{i \to \infty} H_\mu(\alpha \vert \alpha^{D_i \cap \Phi} \vee \AAA) = H_\mu(\alpha \vert \alpha^{\Phi} \vee \AAA).
	\end{equation*}

	 Using the monotone convergence theorem it follows that for sufficiently big $i \in \NN$ 
	\begin{equation}\label{eq:H_D_i}
	\expect_{\tilde \Phi} H_\mu(\alpha \vert \alpha^{D_i \cap \Phi} \vee \AAA) \leq \expect_{\tilde \Phi} H_\mu(\alpha \vert \alpha^{\Phi} \vee \AAA) + \varepsilon.
	\end{equation}
	Choose $i \in \NN$ that satisfies \eqref{eq:H_D_i}, and let $D=D_i$. 
	By monotonicity of conditional entropy (as in  \cite[Proposition $14.18$]{MR1958753}),  for any superset $D' \supset D $ we have
	$$ 	\expect_{\tilde \Phi }H_\mu(\alpha \vert \alpha^{D' \cap \Phi} \vee \AAA) \le
	\expect_{\tilde \Phi }H_\mu(\alpha \vert \alpha^{D \cap \Phi} \vee \AAA).$$

	Thus \eqref{eq:H_D_prime} holds for any   $D' \supset D $. 
	 
\end{proof}

\begin{lem}\label{lem:H_F_nu}
	For every finite $F \subset \Gamma$ we have
	\begin{equation}\label{eq:H_F_nu}
	H_\mu\left(\alpha^F \mid  \AAA \right)=\sum_{g \in F}\expect_{\tilde \Phi}
	H_{\mu}\left(  \alpha \mid \alpha^{F g^{-1} \cap \Phi}\vee \AAA \right).
	\end{equation}
\end{lem}
\begin{proof}
	Fix an instance of $\tilde \Phi:\Gamma \to 2^\Gamma$ that satisfies properties $(i)$--$(iii)$ of an invariant random past. 
	Write $F$ as 
	$F = \{ g_1, \ldots , g_{|F|}\}$ ordered so that $g_j \in \tilde \Phi( g_i)$ iff $i < j$. Applying the chain  rule for entropy (as in  \cite[Proposition $14.18$]{MR1958753}) we have:
		$$H_\mu\left(\alpha^{F} \mid  \AAA \right)=
	\sum_{i=1}^{|F|}  H_{\mu}\left( \alpha^{g_i} \mid \bigvee_{j < i} \alpha^{g_j} \vee \AAA \right) 
	$$
		This can be rewritten as :
		\[
	H_\mu(\alpha^{F} \vert \AAA) = \sum_{g \in F} H_\mu ( \alpha^g \vert  \alpha^{F \cap (\tilde \Phi(g))} \vee \AAA). 
	\]
	Taking the expectation over  $\tilde \Phi$,  using property $(iv)$ of an invariant random  past and linearity of expectation we get \eqref{eq:H_F_nu}.

\end{proof}

\begin{proof}[Proof of Theorem \ref{th: generalized kieffer}]

	Let $(F_i)_{i=1}^\infty$ be a F\o lner sequence in $\Gamma$.

	By Lemma \ref{lem:H_F_nu} we have:
	\begin{equation}\label{eq: partial kieffer}
	H_\mu(\alpha^{F_i}\vert \AAA) = \sum_{g \in F_i} \expect_{\tilde \Phi} H_\mu(\alpha \vert \alpha^{F_i g^{-1} \cap \Phi} \vee \AAA).
	\end{equation}
	Choose any  $\varepsilon>0$, and let $D \subset \Gamma$ as in 
     Lemma \ref{lem:D_approx_H}. Then for any $D' \supset D$ we have  that \[
	\expect_{\tilde \Phi}H_\mu(\alpha \vert \alpha^\Phi \vee \AAA)\leq \expect_{\tilde \Phi}H_\mu(\alpha \vert \alpha^{D' \cap \Phi} \vee \AAA) \leq \expect_{\tilde \Phi}H_\mu(\alpha \vert \alpha^{\Phi} \vee \AAA) + \varepsilon 
	\]
	Let $\partial_D F_i$ denote the set of  $g \in F_i$ such that $D \not\subset F_i g^{-1}$.   
	From the definition of a F\o lner sequence we can derive that 
	\[
	\lim_{i \to \infty} \frac{\lvert \partial_D F_i\rvert}{\lvert F_i\rvert} = 0.
	\]
	For any $g \in F_i\setminus \partial_D F_i$ we will have 
	\[
	\expect_{\tilde \Phi}H_\mu(\alpha \vert \alpha^{\Phi} \vee \AAA) \leq \expect_{\tilde \Phi}H_\mu(\alpha \vert \alpha^{F_i g^{-1} \cap \Phi} \vee \AAA) \leq \expect_{\tilde \Phi}H_\mu(\alpha \vert \alpha^{\Phi} \vee \AAA) + \varepsilon 
	\]
	The latter together with equation \eqref{eq: partial kieffer} implies that
	\[
	\expect_{\tilde \Phi}H_\mu(\alpha \vert \alpha^{\Phi} \vee \AAA)\leq \lim_{i \to \infty} \frac{H_\mu(\alpha^{F_i} \vert \AAA)}{\lvert F_i\rvert} \leq \expect_{\tilde \Phi}H_\mu(\alpha \vert \alpha^{\Phi} \vee \AAA) + \varepsilon.
	\]
	This finishes the proof since $\varepsilon>0$ is arbitrary.
\end{proof}
\begin{cor}\label{cor:amenable_predictability_iff_zero_entropy}
	Let $\Gamma$ be a countable amenable group. Then a probability measure preserving action $\act{\Gamma}{(X,\mu)}$ 
	has zero Kolmogorov-Sinai entropy relative to a factor map $\pi:X \to Y$ if and only if it is measure-theoretically $\nu$-predictable relative to $\pi:X \to Y$ with respect to some (hence any) invariant random  total order.
\end{cor}

\begin{proof}
	An action  $\act{\Gamma}{(X,\mu)}$ has zero Kolmogorov-Sinai entropy if and only if for any finite measurable partition $\alpha$ we have
	$$ h_\Gamma(\alpha, X, \mu\vert \AAA) =0.$$
	By the Kieffer-Pinkser formula, for any invariant random past this is equivalent to 
	$$  \expect_{\tilde \Phi} H_\mu(\alpha \vert \alpha ^{\Phi} \vee \AAA) =0,$$
	which is equivalent to having $H_\mu (\alpha \vert \alpha ^{\Phi} \vee \AAA) =0$ for a.e. realization of $\tilde \Phi$. 
	This is equivalent to the statement that $\alpha$ is measurable with respect to the $\mu$-completion of $\alpha ^{\Phi} \vee \AAA$.
\end{proof}

Furthermore, predictability of a relative generator with respect to an invariant  random partial order implies zero relative entropy:
\begin{prop}\label{thm:measure_predictible_generator_amenable}
	Let $\act{\Gamma}{(X,\mu)}$ be a measure preserving action of a countable amenable group $\Gamma$ and let $\AAA$ be a $\Gamma$-invariant $\sigma$-subalgebra. Suppose that $\prec$ is an invariant random partial order on   $\Gamma$.
	Let $\alpha$ be a finite Shannon entropy partition of $X$ such that $\alpha \subset \AAA \vee \alpha^{\Phi_{\prec}}$ for almost every  instance of $\prec$ and 
	\begin{equation}\label{eq:alpha_relative_generator_def}
	\alpha^\Gamma \vee \AAA = \Borel(X) \mod \mu.
	\end{equation}
	Then
	$h_\Gamma( X, \mu \vert \AAA ) = 0$.
	
\end{prop}
\begin{proof}
	Denote the law of the invariant random partial order $\prec$ by $\nu \in \Prob_\Gamma(\Ord(\Gamma))$.
	By Lemma \ref{lem: good order}, we can find   an invariant random total order $\tilde \nu \in \Prob_\Gamma(\Ord(\Gamma))$ that extends $\nu$. Because $\tilde \nu$ extends $\nu$ it follows that $\alpha \subset \AAA \vee \alpha^{\Phi_{\prec}}$ for $\tilde \nu$ -a.e instance of $\prec$. This implies by Theorem \ref{th: generalized kieffer} that
	$h_\Gamma( X, \mu \vert  \AAA ) = 0$.
\end{proof}

Seward  \cite{1405.3604} defined the \emph{relative Rokhlin entropy} for measure-preserving actions of countable groups.
For an ergodic measure-preserving action $\act{\Gamma}{(X,\mu)}$ of a countable group $\Gamma$ and $\Gamma$-invariant $\sigma$-subalgebra, it is given by
\begin{equation*}
\hRok_\Gamma(X,\mu \mid \AAA) = \inf_\alpha H_\mu(\alpha \mid \AAA),
\end{equation*}
where $\alpha$ ranges over all generating partitions (countable partitions $\alpha$ that satisfy \eqref{eq:alpha_relative_generator_def}).
For a free action of an amenable group, Rokhlin entropy coincides with Kolmogorov-Sinai  entropy.
 Seward proved the following far-reaching extension of Sinai's theorem: Any free ergodic measure-preserving action $\act{\Gamma}{(X,\mu)}$ with positive Rokhlin entropy admits a Bernoulli factor, or equivalently it admits a (non-trivial) partition whose iterates under $\Gamma$ are jointly independent \cite{1804.05269}. 
 As an immediate  corollary of Seward's  Bernoulli factor theorem we have:

\begin{prop}\label{prop:predictabe_zero_roklin_entropy}
	Let $\Gamma$ be a countable group, let  $\nu$ be an invariant random partial order on $\Gamma$. A free ergodic  action  $\act{\Gamma}{(X,\mu)}$ that is measure-theoretically predictable with respect to some  invariant random partial order has zero Rokhlin entropy.
\end{prop}
\begin{proof}
   If a free ergodic action $\act{\Gamma}{(X,\mu)}$ has positive Rokhlin entropy then Seward's Bernoulli factor theorem says that it  admits a non-trivial finite partition $\alpha$ with independent $\Gamma$-iterates. Such a partition $\alpha$ is not measurable with respect to the $\mu$-completion of $\alpha^{\Phi_\prec}$, for any order $\prec$ on $\Gamma$. Hence, a free ergodic action  with positive Rokhlin entropy is not measure-theoretically predictable with respect to any invariant random order.
\end{proof}
\begin{question}
	Is there a direct proof of Proposition \ref{prop:predictabe_zero_roklin_entropy} that does not use  Seward's  Bernoulli factor theorem?
\end{question}

\section{From topological predictability to measure-theoretical predictability via  $\mu$-continuous partitions}\label{sec: continuous partitions}
In this section we complete the proof of Theorem \ref{thm:predicatable_implies_zero_entropy_amenable}.
The steps are essentially identical to  Hochman's proof in  \cite{MR2873161}, where only the lexicographic past on the group  $\ZD$ was considered, without the ``relative'' version (the image of the factor map $\pi: X \to Y$ was the trivial one-point space).
Just as in Hochman's proof, we will rely on the variational principle for topological entropy.  
\begin{thm}[The Variational principle \cite{stepin-tagi-zade,MR533007,MR665415}]\label{thm:variational_principle}
	Let $\act{\Gamma}{X}$ be an action of a countable amenable group by homeomorphisms on a compact metrizable space $X$. 
	Then the topological entropy $h_\Gamma(X)$ is given by
	\begin{equation*}
	h_\Gamma(X) = \sup_{\mu \in \Prob_\Gamma(X)} h_\Gamma(X,\mu).
	\end{equation*}
\end{thm}
Note that Kerr and Li proved a more general variational principal for sofic entropy \cite{MR2854085}.

As in the previous sections,  let $X$ be a compact metric space and let $\mu$ be a Borel probability measure on it.
The Rokhlin distance $d_\mu(\alpha,\beta)$ between two partitions $\alpha$ and $\beta$ of finite Shannon entropy is defined by the formula 
\[
d_\mu(\alpha, \beta) = H_\mu(\alpha \vert \beta) + H_\mu(\beta \vert \alpha). 
\]  
It is well known that for a measure preserving action of an amenable group the Kolmogorov-Sinai entropy is an $1$-Lipschitz function with respect to the Rokhlin metric on the space of partition with finite Shannon entropy.

 A partition $\alpha$ of $(X,\mu)$ is said to be a {\em $\mu$-continuous partition} if there is a continuous function $f: X \to \RR$ such that pieces of $\alpha$ are equal to the level sets of $f$ up to $\mu$-null subset. This definition is due to Hochman \cite{MR2873161}; he proved the following (in the more general setup where $X$ is a normal topological space and  $\mu$ is a regular Borel  probability measure): 
\begin{prop}[Hochman \cite{MR2873161}]
	For any Borel probability measure $\mu \in \Prob(X)$, the  $\mu$-continuous partitions are dense with respect to the Rokhlin metric.
\end{prop}

\begin{proof}[Proof of Theorem \ref{thm:predicatable_implies_zero_entropy_amenable}.]
	Let $\act{\Gamma}{X}$,$\act{\Gamma}{Y}$ and $\pi: X \to Y$ be as in the statement of Theorem \ref{thm:predicatable_implies_zero_entropy_amenable}.
Let $$\AAA = \pi^{-1}(\Borel(Y)) \subset \Borel(X).$$
 Let $\mu$ be a $\Gamma$-invariant measure on $X$.
 We note that $h_\Gamma(Y,\pi(\mu)) = h_\Gamma(\AAA,X,\mu)$.
	Take any $\mu$-continuous partition $\alpha$  of $(X,\mu)$ with $H(\alpha) < \infty$. Then topological predictability of the action relative to $\pi:X \to Y$  implies that  $\alpha$ is measurable with respect to the completion of $\alpha^{\Phi_{\prec}} \vee \AAA$ for almost every  instance of the invariant random total order $\prec$. By Corollary \ref{cor:amenable_predictability_iff_zero_entropy} of Theorem \ref{th: generalized kieffer} 
	it follows that
	$$h_\Gamma(\alpha,X,\mu \mid \AAA) =0.$$
	Since this holds for a set of partitions that is dense with respect to the Rokhlin metric,
	and the function $\alpha \mapsto h_\Gamma(\alpha,X,\mu \mid \AAA)$ is continuous (in fact $1$-Lipschitz with respect to the Rokhlin-metric), it follows that
	\begin{equation*}
	h_\Gamma(X,\mu \mid \AAA) = 0.
	\end{equation*}
	
	The  Abramov-Rokhlin entropy addition formula for amenable group actions \cite{ward_amenable_abramov_formula} asserts that  
	\[
	h_\Gamma( X, \mu ) = h_\Gamma( \AAA,X, \mu) + h_\Gamma(X,\mu \mid \AAA).
	\]
	Thus for any $\mu \in\Prob_\Gamma(X)$:
	$$
	h_\Gamma(X,\mu)= h_\Gamma(\AAA,X,\mu)= h_\Gamma(Y, \pi(\mu)). 
	$$
	By the variational principle (Theorem \ref{thm:variational_principle}) it follows that 
	\[
	h_\Gamma(X) \leq h_\Gamma(Y).
	\]
	Since the topological entropy for actions of amenable groups is factor-monotone, we have
	\[
	h_\Gamma(X) = h_\Gamma(Y).
	\]
\end{proof}

\section{Prime actions have zero topological entropy}\label{sec:prime_actions}

An action $\act{\Gamma}{X}$ is called \emph{topologically prime} if every factor map is either an isomorphism or it maps onto the trivial action on the one-point space.
More generally, if $\pi:X \to Y$ is a topological factor map between  $\act{\Gamma}{X}$ and  $\act{\Gamma}{Y}$, we say that   $\pi:X \to Y$ is a \emph{topologically prime extension} if $\act{\Gamma}{X}$ has no intermediate factors.
Equivalently, $C(X)$ has no strict $\Gamma$-invariant  $C^*$-subalgebras that strictly contain $\pi_*(C(Y))$.
King constructed and example of a homeomorphism on the Cantor set that is  topologically prime \cite{MR1091424}, because it has ``topological minimal self-joinings'' in the sense of del Junco \cite{MR896794}. The later property, called ``doubly minimal'' by Weiss  \cite{MR1603185},  means that any pair of points $x,y \in X$ that are not in the same orbit have a dense orbit in $X \times X$.  
More generally, any free and ergodic measure-preserving $\ZZ$-action with zero entropy  admits a uniquely ergodic  doubly minimal  model  \cite{MR1603185}. In particular there  is topologically prime  model for any free, ergodic $\ZZ$-action with zero entropy \cite[Theorem $13.1$]{MR2186250}.


The following result was communicated to us by Benjy Weiss. With his kind permission we reproduce his proof. 
\begin{thm}\label{thm:prime_extensions_have_zero_entropy}
	Suppose that $\Gamma$ is a countable amenable group that acts on $X$ and $Y$, and that $\pi:X \to Y$ is a topologically prime extension.
	Then $h_\Gamma(X)=h_\Gamma(Y)$.
\end{thm}
For $\Gamma=\ZZ$ and $\pi$ equal to the trivial factor, a proof of Theorem \ref{thm:prime_extensions_have_zero_entropy} appears in \cite{MR1125888}. See also \cite[Section $13$]{MR2186250}.

\begin{proof}
Let $\pi: X \to Y$ be a topologically prime extension. Choose any 
$\Gamma$-invariant measure $\mu \in  \Prob_\Gamma(X)$, and let $\nu \in \Prob_\gamma(Y)$ be the push-forward of $\mu$ via $\pi$. We will prove that 
\begin{equation}\label{eq: measure_entropy_equality_ergodic}
h_\Gamma(X,\mu)=h_\Gamma(Y,\nu).
\end{equation}
We first prove this under the additional assumption that $\mu$ satisfies the following property: 
\begin{equation}\label{eq:mu_cont_not_isomorphism}
\inf \left\{ \mu(A) :~ A \mbox { is not contained in the  } \mu \mbox{-completion of } \pi^{-1}(\Borel(Y))\right\} =0.
\end{equation}

Since $\mu$-continuous partitions are dense with respect to the Rokhlin metric, we can find for any $\epsilon >0$ a $\mu$-continuous partition $\alpha$  such that 
\begin{equation}\label{eq:H_mu_alpha_mid_Y}
0 <H_\mu(\alpha \mid \pi^{-1}\Borel(Y)) < \epsilon.
\end{equation}

It follows that
\begin{equation*} 
h_\mu(\alpha, X , \mu \mid \pi^{-1}\Borel(Y))  < \epsilon.
\end{equation*}  
Since $\pi:X \to Y$ is  a topologically prime extension, for any $f \in C(X) \setminus \pi_*(C(Y))$,
 $C(X)$ is contained in the $\Gamma$-invariant $C^*$-algebra generated by $f$ and $\pi_*(C(Y))$.
By the left inequality in \eqref{eq:H_mu_alpha_mid_Y}, $\alpha$  is not contained in the $\mu$-completion of $\pi^{-1}(\Borel(Y))$. It follows that 
\begin{equation*} 
(\alpha \vee \pi^{-1}\Borel(Y))^\Gamma = \Borel(X) \mod \mu.
\end{equation*}
%
By Kolmogorov-Sinai theorem,
$$h_\Gamma(X,\mu \mid \pi^{-1}\Borel(Y))  < \epsilon.$$
Since $\epsilon >0$ was arbitrary, this shows
that
$$h_\Gamma(X,\mu \mid \pi^{-1}\Borel(Y))  =0.$$
By the Abramov-Rokhlin entropy addition formula, this implies \eqref{eq: measure_entropy_equality_ergodic}.

It remains to prove that \eqref{eq: measure_entropy_equality_ergodic} holds even when $\mu \in \Prob_\Gamma(X)$ does not satisfy \eqref{eq:mu_cont_not_isomorphism}. 
	If \eqref{eq:mu_cont_not_isomorphism} does not hold, then there exists $\epsilon >0$ so that every $A \in \Borel(X)$ with $\mu(A) <\epsilon$ is also measurable with respect to the  $\mu$-completion of $\pi^{-1}\Borel(Y)$. In this case the factor map $\pi:X \to Y$  is of the following very degenerate form:
Let $\mu =p\cdot \mu_c + (1-p)\cdot \mu_a$ be the representation of  $\mu$ as a convex combination of a  purely continuous measure $\mu_c \in \Prob_\Gamma(X)$  and purely atomic measure $\mu_a \in \Prob_\Gamma(X)$. Then $\pi:X \to Y$ gives a measure-preserving isomorphism between $\act{\Gamma}{(X,\mu_c)}$ and $\act{\Gamma}{\left(Y,\pi(\mu_c)\right)}$.
In particular,  the factor map $\pi:X \to Y$  is finite-to-one (actually bounded-to-one) $\mu$-almost everywhere. Because $\Gamma$ is an infinite amenable group,  finite-to-one extensions do not increase entropy and  \eqref{eq: measure_entropy_equality_ergodic} follows in this case.

\end{proof}

\begin{remark}
	In the proof above we use the fact that $\Gamma$ is infinite to conclude that    finite-to-one extensions do not increase entropy.
	If $\Gamma$ is a finite group, Theorem \ref{thm:prime_extensions_have_zero_entropy} fails. To see this, take $\act{\Gamma}{X}$ to be the action of a finite group $\Gamma$ on itself by translations, and take $\act{\Gamma}{Y}$ to be an action of $\Gamma$ on the cosets of a maximal proper subgroup.
\end{remark}

We  now consider topologically prime actions of sofic groups. For a measure preserving action $\act{\Gamma}{(X,\mu)}$ of a  sofic group $\Gamma$ with a sofic approximation $\Sigma$, let $h_\Gamma^\Sigma(X,\mu)$ denote the sofic entropy of the action (for definitions see for instance \cite{1711.02062}).

The first-named author and Brandon Seward  \cite[Proposition $8.8$]{1705.09707} and independently Ben Hayes \cite[Proposition  $2.7$ (i)]{hayes2016relative}  proved  the following Abramov-Rokhlin-type inequality:

\begin{prop}[Abramov-Rokhlin sub-addition formula for sofic entropy]\label{prop:abramov_rokhlin_sodic}
	Let $\Gamma$ be a sofic group with sofic approximation $\Sigma$, and let $\pi:X \to Y$ be a factor map between the measure preserving actions  $\act{\Gamma}{(X,\mu)}$ and $\act{\Gamma}{(Y,\mu)}$. If $\alpha$ is a measurable partition of $X$ such that
	\begin{equation}\label{eq:alpha_relative_generator}
	(\alpha \vee \pi^{-1}(\Borel(Y)))^\Gamma \supset \Borel(X) \mod \mu,
	\end{equation}
	then
	$$h_\Gamma^\Sigma(X,\mu) \le h_\Gamma^\Sigma(Y,\nu) + H_\mu(\alpha \mid \pi^{-1}(\Borel(Y)).$$
\end{prop}
The above  is a simplified and slightly less general form of the corresponding statements from \cite{1705.09707,hayes2016relative}.

The following proposition on sofic entropy of finite-to-one extension follows from the main theorem of Hayes' work \cite{MR3635672} relating sofic entropy and spectral properties of actions. That work establishes much more general results. In particular, the statement below will hold with compact extensions instead of finite-to-one.  See the discussion following Proposition $5.7$ in \cite{1704.06349}. The case needed in our exposition allows for a short combinatorial proof, which we include for the sake of self-containment.

\begin{prop}\label{prop:finite_to_one_sofic_entropy}
	Let $\Gamma$ be a countably infinite sofic group with sofic approximation $\Sigma$, and let $\pi:X \to Y$ be a finite-to-one factor map between the measure preserving actions  $\act{\Gamma}{(X,\mu)}$ and $\act{\Gamma}{(Y,\nu)}$. Then 
    \begin{equation}\label{eq:h_mu_sofic_X_Y}
	h_\Gamma^\Sigma(X,\mu)\le h_\Gamma^\Sigma(Y,\nu).
	\end{equation}
\end{prop}
\begin{proof}
	Because $\pi:X \to Y$ is finite-to-one, by removing  null sets we can assume that  $|\pi^{-1}(\{y\})|$ is a positive integer for every $y \in Y$.
	For every $n \in \NN$ let $$Y_n = \{y \in Y:~ \pi^{-1}(\{y\}) =n\} \mbox{ and } X_n = \pi^{-1}(Y_n).$$
	
Let us first assume in addition that $\act{\Gamma}{(Y,\nu)}$ is aperiodic.
By the Abramov-Rokhlin sub-addition formula for sofic entropy (Proposition \ref{prop:abramov_rokhlin_sodic}) it is enough to show that for every $\epsilon >0$ we can find a measurable partition $\alpha$ of $X$ that satisfies \eqref{eq:alpha_relative_generator}
and $H_\mu(\alpha) < \epsilon$.
Because $\act{\Gamma}{(Y,\nu)}$ is aperiodic, so is  $\act{\Gamma}{(X,\nu)}$. So $(X,\mu)$ is a standard probability space with no atoms.
Let $\phi:X \to [0,1]$ be a Borel bijection. 

Because $\act{\Gamma}{(Y,\nu)}$ is aperiodic  whenever $\nu(Y_n) >0$ we can find arbitrary small $\epsilon_n >0$ and  a Borel measurable set $A_n \subset Y_n$ so that $\bigcup_{g \in\Gamma}g(A_n)= Y_n$ and  $\nu(A_n)=\epsilon_n$ (for instance by considering the ergodic decomposition of $\act{\Gamma}{(Y,\nu)}$). 
Because $\pi:X \to Y$ is finite-to-one  we can define $\psi:X \to \NN$ by
$$\psi(x)= \# \left\{x' \in \pi^{-1}\left(\{\pi(x)\}\right):~ \phi(x')  \le \phi(x) \right\}.$$
Then for every $y \in Y_n$,  $\psi$ induces a bijection between  $\pi^{-1}\left(\{y\}\right)$ and $\{1,\ldots,n\}$.
We denote the  inverse  by $$\psi_{y}^{-1}:\{1,\ldots,n\} \to \pi^{-1}\left(\{y\}\right).$$
Also we have  a Borel bijection $\Psi:X \to \bigcup_{i=1}^\infty \left(Y_n \times \{1,\ldots,n\}\right)$
given by
$$\Psi(x)= (\pi(x),\psi(x)).$$
For $y \in Y_n$ and $g \in \Gamma$ define a permutation $\Pi_{g,y}$ of $\{1,\ldots,n\}$
by
$$\Pi_{g,y}(i) = \psi\left(g(\psi_{y}^{-1}(i))\right) \mbox{ for } i \in \{1,\ldots,n\}.$$
Then the map $(g,y) \mapsto \Pi_{g,y}$ is a Borel map from $\Gamma \times Y_n$ to the set of permutations on $\{1,\ldots,n\}$.
For every $n \in \NN$ and $i \in \{1,\ldots,n\}$ let 
$$B_{n,i}= \pi^{-1}(A_n) \cap \psi^{-1}(\{i\}).$$
Consider the partition $\alpha$ of $X$ given by
$$\alpha =\left\{ B_{n,i}:~ n \in \NN \mbox{ and } 1 \le i \le n \right\} \cup \{X \setminus \bigcup_{n \in \NN}\pi^{-1}(A_n)\}.$$
It follows that  
$$H_\mu(\alpha) 
< \sum_{n=1}^{\infty}\epsilon_n\left(\log(\epsilon_n^{-1})+\log(n) \right)+\log\left(1-\sum_{n=1}^\infty \epsilon_n\right),$$
So if $(\epsilon_n)_{n=1}^\infty$ are sufficiently small $H_\mu(\alpha) < \epsilon$.
In order to check that $\alpha$ is a relative generator in the sense that it satisfies \eqref{eq:alpha_relative_generator}, it suffices to check that 
$\psi:X \to \NN$ is $ \alpha^\Gamma \vee \pi^{-1}(\Borel(Y))$-measurable. Indeed, because $Y_n = \bigcup_{g \in \Gamma}g (A_n)$ for any $x \in X_n$ there exists $g \in \Gamma$ and $i  \in \{1,\ldots,n\}$ so that $g(x) \in B_{n,i}$. For such $x$ we have
 $\psi(g(x))=i$ so $\psi^{-1}_{g(\pi(x))}(i)=g(x)$ so 
$$\psi(x)= \psi(g^{-1}(g(x)))=\psi(g^{-1}(\psi^{-1}_{g(\pi(x))}(i)))=\Pi_{g^{-1},\pi(g \cdot x)}(i).$$
This implies the desirable measurability.

Now let us remove the additional assumption that $\act{\Gamma}{(Y,\nu)}$ is aperiodic.
Let $\act{\Gamma}{(Z,\eta)}$ be a Bernoulli action.
Then $\pi:X \to Y$ naturally induces a finite-to-one factor map from $\act{\Gamma}{(X \times Z,\mu \times \eta)}$ to $\act{\Gamma}{(Y \times Z,\nu \times \eta)}$. Since $\act{\Gamma}{(Y \times Z,\nu \times \eta)}$ is aperiodic we can use the first part to conclude that 
$$h_\Gamma^\Sigma(X \times Z, \mu \times \eta) \le h_\Gamma^\Sigma(Y \times Z, \mu \times \eta).$$
By  \cite[Theorem $8.1$]{MR2552252}, because $(Z,\eta)$ is a  Bernoulli action
\begin{equation*}\label{eq:sofic_entropy_bernoulli_addition_formula}
h_\Gamma^\Sigma(X \times Z, \mu \times \eta)= h_\Gamma^\Sigma(X,\mu)+h_\Gamma^{\Sigma}(Z,\eta),
\end{equation*}
and
\begin{equation*}
h_\Gamma^\Sigma(Y \times Z, \mu \times \eta)= h_\Gamma^\Sigma(Y,\mu)+h_\Gamma^{\Sigma}(Z,\eta).
\end{equation*}
This proves \eqref{eq:h_mu_sofic_X_Y} without assuming that  $\act{\Gamma}{(Y,\nu)}$ is aperiodic.
\end{proof}

\begin{thm}\label{thm:prime_extensions_sofic_entropy}
	Suppose $\Gamma$ is a sofic group with sofic approximation $\Sigma$ that acts on $X$ and $Y$ and that  $\pi:X \to Y$ is a topologically prime extension.
	 Then 
	\begin{equation}\label{eq:h_sofic_X_Y}
	h_\Gamma^\Sigma(X)\le h_\Gamma^\Sigma(Y).
	\end{equation}
	In particular, topologically prime actions of sofic groups have non-positive sofic entropy.
\end{thm}
\begin{proof}
We obtain the proof essentially by repeating the proof of Theorem \ref{thm:prime_extensions_have_zero_entropy}, with the following modifications:
Instead of applying the Abramov-Rokhlin entropy addition formula we apply the  Abramov-Rokhlin sub-addition formula for sofic entropy (Proposition \ref{prop:abramov_rokhlin_sodic}). To deal with the case where \eqref{eq:mu_cont_not_isomorphism} fails, we apply Proposition \ref{prop:finite_to_one_sofic_entropy} above.
\end{proof}


We note that in general we cannot  conclude an equality instead of  the equality \eqref{eq:h_sofic_X_Y} under the assumptions of Theorem \ref{thm:prime_extensions_sofic_entropy}. 
From the definition of sofic entropy, if $\act{X}{\Gamma}$ admits no invariant probability measures then $h_\Gamma^\Sigma(X) = - \infty$.  
The following is an example of a topologically prime action of the free group on two generators that admits no invariant probability measure:

\begin{example}
	Let $T_1:X \to X$ be a topologically prime homeomorphism that is uniquely ergodic (for instance Kings's example \cite{MR1091424} on the cantor set $X$), and let $T_2:X \to X$ be a homeomorphism that does not preserve the unique $T_1$-invariant measure. Consider the action on $X$ of the free group generated by $T_1$ and $T_2$. This is a topologically prime action because it has $T_1$ as a subaction. It  admits no invariant probability measure, because the unique $T_1$-invariant measure is not $T_2$-invariant.
\end{example}

\begin{remark}
	Although finite-to-one  extensions can never increase sofic entropy  they can certainly decrease it, as in  the well known Ornstein-Weiss example of a two-to-one factor map from the Bernoulli $2$-shift to the Bernoulli $4$-shift over the free group. 
\end{remark}
\begin{remark}
	Using the Abramov-Rokhlin sofic entropy sub-addition formula and similar arguments  as in the proof of Proposition \ref{prop:finite_to_one_sofic_entropy} it is possible  to prove the following ``atomless'' refinement of the Abramov-Rokhlin entropy sub-addition formula as follows: If $\mu= p \mu_c + (1-p)\mu_a$ is the representation of $\mu$ as a convex combination of a continuous measure $\mu_c$ and a purely atomic measure $\mu_a$ and
    $\alpha$ is a measurable partition of $X$ such that
	$$(\alpha \vee \pi^{-1}(\Borel(Y)))^\Gamma \supset \Borel(X) \mod \mu_c,$$
	then
	$$h_\Gamma^\Sigma(X,\mu) \le h_\Gamma^\Sigma(Y,\nu) + H_{\mu}(\alpha \mid \pi^{-1}(\Borel(Y)).$$
\end{remark}

\bibliographystyle{abbrv}
\bibliography{predictability_groups}

\begin{thebibliography}{10}

\bibitem{1705.08559}
A.~Alpeev.
\newblock A random ordering formula for sofic and {Rokhlin} entropy of {Gibbs}
  measures.
\newblock {\em preprint, arXiv:1705.08559}, 2017.

\bibitem{1705.09707}
A.~Alpeev and B.~Seward.
\newblock Krieger's finite generator theorem for actions of countable groups
  {III}.
\newblock {\em preprint, arXiv:1705.09707}, 2017.

\bibitem{austin-podder}
T.~Austin and M.~Podder.
\newblock {Gibbs} measures over locally tree-like graphs and percolative
  entropy over infinite regular trees.
\newblock {\em Journal of Statistical Physics}, 170(5):932--951, 2018.

\bibitem{MR2552252}
L.~Bowen.
\newblock Measure conjugacy invariants for actions of countable sofic groups.
\newblock {\em J. Amer. Math. Soc.}, 23(1):217--245, 2010.

\bibitem{1711.02062}
L.~Bowen.
\newblock A brief introduction to sofic entropy theory.
\newblock {\em preprint, arXiv:1711.02062}, 2017.

\bibitem{1704.06349}
L.~Bowen.
\newblock Examples in the entropy theory of countable group actions.
\newblock {\em preprint, arXiv:1704.06349}, 2017.

\bibitem{MR896794}
A.~del Junco.
\newblock On minimal self-joinings in topological dynamics.
\newblock {\em Ergodic Theory Dynam. Systems}, 7(2):211--227, 1987.

\bibitem{1408.5805}
B.~Deroin, A.~Navas, and C.~Rivas.
\newblock Groups, orders, and dynamics, 2014.

\bibitem{MR0327605}
E.~Formanek.
\newblock Extending partial right orders on nilpotent groups.
\newblock {\em J. London Math. Soc. (2)}, 7:131--134, 1973.

\bibitem{MR1958753}
E.~Glasner.
\newblock {\em Ergodic theory via joinings}, volume 101 of {\em Mathematical
  Surveys and Monographs}.
\newblock American Mathematical Society, Providence, RI, 2003.

\bibitem{MR2186250}
E.~Glasner and B.~Weiss.
\newblock On the interplay between measurable and topological dynamics.
\newblock In {\em Handbook of dynamical systems. {V}ol. 1{B}}, pages 597--648.
  Elsevier B. V., Amsterdam, 2006.

\bibitem{hayes2016relative}
B.~Hayes.
\newblock Relative entropy and the pinsker product formula for sofic groups.
\newblock {\em arXiv preprint arXiv:1605.01747}, 2016.

\bibitem{MR3635672}
B.~Hayes.
\newblock Mixing and spectral gap relative to {P}insker factors for sofic
  groups.
\newblock In {\em Proceedings of the 2014 {M}aui and 2015 {Q}inhuangdao
  conferences in honour of {V}aughan {F}. {R}. {J}ones' 60th birthday},
  volume~46 of {\em Proc. Centre Math. Appl. Austral. Nat. Univ.}, pages
  193--221. Austral. Nat. Univ., Canberra, 2017.

\bibitem{MR2873161}
M.~Hochman.
\newblock On notions of determinism in topological dynamics.
\newblock {\em Ergodic Theory Dynam. Systems}, 32(1):119--140, 2012.

\bibitem{MR3530053}
W.~Huang, L.~Jin, and X.~Ye.
\newblock On extensions over semigroups and applications.
\newblock {\em Entropy}, 18(6):Paper No. 230, 6, 2016.

\bibitem{MR3296581}
W.~Huang, L.~Xu, and Y.~Yi.
\newblock Asymptotic pairs, stable sets and chaos in positive entropy systems.
\newblock {\em J. Funct. Anal.}, 268(4):824--846, 2015.

\bibitem{MR2025310}
B.~Kami\'nski, A.~Siemaszko, and J.~Szyma\'nski.
\newblock The determinism and the {K}olmogorov property in topological
  dynamics.
\newblock {\em Bull. Polish Acad. Sci. Math.}, 51(4):401--417, 2003.

\bibitem{MR2162751}
B.~Kami\'{n}ski, A.~Siemaszko, and J.~Szyma\'{n}ski.
\newblock Extreme relations for topological flows.
\newblock {\em Bull. Pol. Acad. Sci. Math.}, 53(1):17--24, 2005.

\bibitem{MR2854085}
D.~Kerr and H.~Li.
\newblock Entropy and the variational principle for actions of sofic groups.
\newblock {\em Invent. Math.}, 186(3):501--558, 2011.

\bibitem{MR3616077}
D.~Kerr and H.~Li.
\newblock {\em Ergodic theory}.
\newblock Springer Monographs in Mathematics. Springer, Cham, 2016.
\newblock Independence and dichotomies.

\bibitem{MR0393422}
J.~C. Kieffer.
\newblock A generalized {S}hannon-{M}c{M}illan theorem for the action of an
  amenable group on a probability space.
\newblock {\em Ann. Probability}, 3(6):1031--1037, 1975.

\bibitem{MR1091424}
J.~L. King.
\newblock A map with topological minimal self-joinings in the sense of del
  {J}unco.
\newblock {\em Ergodic Theory Dynam. Systems}, 10(4):745--761, 1990.

\bibitem{kolmogorov}
A.~N. Kolmogorov.
\newblock A new metric invariant of transient dynamical systems and
  automorphisms in {Lebesgue} spaces.
\newblock {\em Doklady Akademii Nauk SSSR}, 119(5):861--864, 1958.

\bibitem{MR3231224}
P.~Linnell and D.~W. Morris.
\newblock Amenable groups with a locally invariant order are locally indicable.
\newblock {\em Groups Geom. Dyn.}, 8(2):467--478, 2014.

\bibitem{MR2286034}
D.~W. Morris.
\newblock Amenable groups that act on the line.
\newblock {\em Algebr. Geom. Topol.}, 6:2509--2518, 2006.

\bibitem{MR533007}
J.~Moulin~Ollagnier and D.~Pinchon.
\newblock Groupes pavables et principe variationnel.
\newblock {\em Z. Wahrsch. Verw. Gebiete}, 48(1):71--79, 1979.

\bibitem{MR665415}
J.~Moulin~Ollagnier and D.~Pinchon.
\newblock The variational principle.
\newblock {\em Studia Math.}, 72(2):151--159, 1982.

\bibitem{MR767264}
J.-P. Pier.
\newblock {\em Amenable locally compact groups}.
\newblock Pure and Applied Mathematics (New York). John Wiley \& Sons, Inc.,
  New York, 1984.
\newblock A Wiley-Interscience Publication.

\bibitem{MR0311538}
A.~H. Rhemtulla.
\newblock Right-ordered groups.
\newblock {\em Canad. J. Math.}, 24:891--895, 1972.

\bibitem{1405.3604}
B.~Seward.
\newblock Krieger's finite generator theorem for actions of countable groups i.
\newblock {\em preprint, arXiv:1405.3604}, 2014.

\bibitem{1602.06680}
B.~Seward.
\newblock Weak containment and {Rokhlin} entropy.
\newblock {\em preprint, arXiv:1602.06680}, 2016.

\bibitem{1804.05269}
B.~Seward.
\newblock Positive entropy actions of countable groups factor onto bernoulli
  shifts.
\newblock {\em preprint, arXiv:1804.05269}, 2018.

\bibitem{MR1125888}
M.~Shub and B.~Weiss.
\newblock Can one always lower topological entropy?
\newblock {\em Ergodic Theory Dynam. Systems}, 11(3):535--546, 1991.

\bibitem{stepin-tagi-zade}
A.~M. Stepin and A.~T. Tagi-Zade.
\newblock Variational characterization of topological pressure of the amenable
  groups of transformations.
\newblock {\em Doklady Akademii Nauk SSSR}, 254(3):545--549, 1980.

\bibitem{ward_amenable_abramov_formula}
T.~Ward and Q.~Zhang.
\newblock The {A}bramov-{R}okhlin entropy addition formula for amenable group
  actions.
\newblock {\em Monatsh. Math.}, 114(3-4):317--329, 1992.

\bibitem{MR1603185}
B.~Weiss.
\newblock Multiple recurrence and doubly minimal systems.
\newblock In {\em Topological dynamics and applications ({M}inneapolis, {MN},
  1995)}, volume 215 of {\em Contemp. Math.}, pages 189--196. Amer. Math. Soc.,
  Providence, RI, 1998.

\end{thebibliography}
\end{document}